\documentclass[11pt]{amsproc}
\usepackage{mathrsfs}
\usepackage{amsfonts}
\usepackage{xy}
\usepackage{amsmath,amstext,amsbsy,amssymb}
\usepackage{color}

\newtheorem{theorem}{Theorem}[section]
\newtheorem{lemma}[theorem]{Lemma}

\newtheorem{proposition}[theorem]{Proposition}
\newtheorem{corollary}[theorem]{Corollary}

\theoremstyle{definition}

\theoremstyle{remark}

\numberwithin{equation}{section}

\begin{document}
\title[Vertex representations of toroidal Lie superalgebras]{Vertex representations of toroidal special linear Lie
superalgebras}
\author{Naihuan Jing}
\address{
Department of Mathematics, North Carolina State University,
Raleigh, NC 27695, USA and
School of Mathematical Sciences, South China University of Technology, Guangzhou, Guangdong 510640, China} \email{jing@math.ncsu.edu}
\author{Chongbin Xu*}
\address{School of Mathematical Sciences, South China University of Technology, Guangzhou, Guangdong 510640
and School of Mathematics \& Information, Wenzhou University, Wenzhou,
Zhejiang 325035, China} \email{xuchongbin1977@126.com}
\thanks{*Corresponding author}\thanks{Jing gratefully acknowledges the support from
Simons Foundation (No.198219) and NSFC (No.11271138)} 

\begin{abstract} Based on the loop-algebraic presentation of 2-toroidal Lie
superalgebras, free field representation of toroidal Lie
superalgebras of type $A(m, n)$ is constructed using both vertex
operators and bosonic fields.
\end{abstract} \maketitle

\section{Introduction}

Let $\mathfrak{g}$ be a finite-dimensional simple Lie (super)algebra
of type $X$ and $R$ be the ring of Laurent polynominals in commuting
variables, the toroidal Lie (super)algebra $T(\mathfrak g)$ is by
definition the perfect central extension of the loop algebra
$\mathfrak{g}\otimes R$. When $R=\mathbb C[t, t^{-1}]$, the toroidal
Lie algebra is the affine Kac-Moody Lie algebra. The larger class
of Lie (super)algebras $T(\mathfrak g)$ shares many properties
with the untwisted affine Lie (super)algebras.

In the case of untwisted toroidal Lie algebras, Moody, Rao and
Yokonuma \cite{MRY} gave a loop algebra presentation for the
2-toroidal Lie algebras similar to the affine Kac-Moody Lie
algebras, which has set the stage for later developments such as
free field realizations and vertex operator representations. Notably
in \cite{T} the type $B_{n}$ toroidal Lie algebras were constructed
using fermionic operators (also cf. \cite{JMe}). On the other hand
level one representations of toroidal Lie algebras of simply laced
types were realized via McKay correspondence and wreath products of
 Kleinian subgroups of $SL_2(\mathbb C)$ \cite{FJW}.
 Bosonic realizations of higher level toroidal Lie
 algebras $T(A_1)$ were also given in \cite{JMT}. More recently
 a unified realization \cite{JM, JMX}
of all 2-toroidal Lie algebras of classical types was constructed
 using bosonic or ferminoic fields, which has generalized
the Feingold-Frenkel construction \cite{FF} for affine Lie algebras.

Affine Lie superalgebras have been studied as early as their
non-super counterparts. In fact, Feingold and Frenkel construction
works for Lie superalgebras as well \cite{FF}. Vertex superalgebras
and their representations were also given in \cite{KWn}. Later in
\cite{KWk} integrable highest weight modules were constructed for
affine superalgebras of orthosymplectic seises using fermionic and
bosonic fields. All these constructions were based on the loop
algebra realizations of affine Lie superalgebras.

Irreducible highest weight modules of classical toroidal Lie
superalgebras can be constructed abstractly as in the affine cases
\cite{RZ}. Various other constructions of toroidal Lie superalgebras
and their generalizations were known \cite{RM, LH, JM, L, BCG}. In
particular, \cite{R} has constructed certain vertex operator
representation for the general toroidal cases. Recently we have
given a loop algebra realization for 2-toroidal classical
superalgebras \cite{JX}, which is a super analog of the MRY
construction (see \cite{IK} for earlier development).

The aim of this work is to generalize Kac and Wakimoto's work on
affine superalgebras of unitary seises to 2-toroidal setting using
both vertex operators and Weyl bosonic fields, and the construction
has utilized our recent MRY presentation exclusively. We remark
that our work is different from  \cite{R} in that we use more
bosonic fields while the latter used more vertex operators. This
suggests that there could be a super boson-fermion correspondence
for the 2-toroidal cases.

The paper is organized as follows. In section 2 we recall the notion
of 2-toroidal Lie superalgebras and the loop-algebra presentation.
In section 3 we construct certain vertex operators and Weyl bosonic
fields to give a level one representation of the 2-toroidal Lie
special linear superalgebra.

\section{Toroidal Lie superalgebra $\mathfrak{T}(A(m,n))$}
Let $V=\mathbb C^{m|n+1}$ be the $\mathbb Z_2$-graded vector space
of dimension $(m, n+1)$, where $m\neq n$. Let $\mathfrak{gl}(m|n+1)$
be the Lie superalgebra of the super-endomorphisms of $V$ under the superbaracket.
Let $\mathfrak{g}$ be the traceless subalgebra, i.e. the
simple Lie superalgebra of type $A(m, n)$.
Let $R=\mathbb C[s^{\pm1},t^{\pm1}]$ be the complex commutative ring
of Laurant polynomials in $s, t$. The loop Lie
superalgebra $L(\mathfrak g):=\mathfrak g\otimes R$ is defined under
the Lie superbracket
$
[x\otimes a, y\otimes b]=[x, y]\otimes ab.
$

Let $\Omega_R$ be the $R$-module of K\"ahler differentials
$\{bda |a, b\in R\}$, and let $d\Omega_R$ be the space of exact
forms. The quotient space $\Omega_R/d\Omega_R$ has a basis consisting of
$\overline{s^{m-1}t^{n}ds}$, $\overline{s^{n}t^{-1}dt}$,
$\overline{s^{-1}ds}$, where $m, n\in\mathbb Z$. Here $\overline{a}$
denotes the coset $a+d\Omega_R$.

The toroidal special linear superalgebr $T(\mathfrak g)$ is defined
to be the Lie superalgebra on the following vector space:
\begin{equation*}
T({\mathfrak g})={\mathfrak g}\otimes R\oplus \Omega_R/d\Omega_R
\end{equation*}
with the Lie superbracket ($x,y\in \mathfrak{g},~a,b\in R$):
$$[x\otimes a, y\otimes b]=[x, y]\otimes ab+(x|y)\overline{(da)b},
\quad [T({\mathfrak g}), \Omega_R/d\Omega_R]=0
$$ and the parities
are specified by: $$p(x\otimes a)=p(x),\quad
p(\Omega_R/d\Omega_R)=\overline{0}.$$

Let $A=(a_{ij})$ be the extended distinguished Cartan matrix of the
affine Lie superalgebra of type $A(m,n)^{(1)}$, i.e.
$$
\begin{pmatrix}
 2 & -1& 0 &\cdots&0&\cdots&0 &1 \\
-1 & 2 & -1& \ddots&0&\cdots&0 &0 \\
\vdots&\ddots&\ddots&\ddots&\ddots&\cdots&\cdots&\vdots\\
0 &\cdots  & -1  & 0 &1&\cdots&0&0 \\
\vdots&\vdots&\vdots&\ddots&\ddots&\ddots&\vdots&\vdots\\
0& \cdots & 0&  &&2&-1 & 0 \\
0 &\cdots  &0&  & &-1&2 & -1 \\
-1 & \cdots &0&0&\cdots&0&-1 & 2
\end{pmatrix}$$
and $Q=\mathbb{Z}\alpha_{0}\oplus\cdots\oplus
\mathbb{Z}\alpha_{m+n+1}$ be its root lattice. Here
$\alpha_{0},\alpha_{m+1}$ are the odd simple roots. The standard
invariant form is given by $(\alpha_{i},\alpha_{j})=d_{i}a_{ij}$,
where
$(d_{0},d_{1},\cdots,d_{m+n+1})=(1,\underbrace{1,\cdots,1}_{m},\underbrace{-1,\cdots,-1}_{n+1})$.

We first recall the loop algebra presentation of the 2-toroidal Lie superalgebras.
\begin{theorem} \cite{JX} The  toroidal special linear superalgebra $T(\mathfrak{g})$
is isomorphic to the Lie superalgebra
 $\mathfrak{T}(A(m,n))$  generated by
$$\{\mathcal{K},\alpha_{i}(k),x^{\pm}_{i}(k)|\, 0\leqslant i\leqslant
m+n+1,k\in\mathbb{Z}\}$$
 with parities given as : \emph{(}$0\leqslant i\leqslant m+n+1,k\in\mathbb{Z}$\emph{)}
$$p(\mathcal{K})=p(\alpha_{i}(k))=\overline{0},\quad p(x^{\pm}_{i}(k))=p(\alpha_{i}).$$
The defining relations of superbrackets are given by:
\begin{eqnarray*}
   &1)&[\mathcal{K},\alpha_{i}(k)]=[\mathcal{K},x^{\pm}_{i}(k)]=0; \\
   &2)& [\alpha_{i}(k),\alpha_{j}(l)]=k(\alpha_{i}|\alpha_{j})\delta_{k,-l}\mathcal{K}; \\
   &3)& [\alpha_{i}(k),x^{\pm}_{j}(l)]=\pm(\alpha_{i}|\alpha_{j})x^{\pm}_{j}(k+l); \\
   &4)& [x^{+}_{i}(k),x^{-}_{j}(l)]=0,\mbox{\emph{if}}~i\neq j;\\
   &&[ x^{+}_{i}(k),x^{-}_{i}(l)]=-\{\alpha_{i}(k+l)
   +k\delta_{k,-l}\mathcal{K}\},\mbox{\emph{if}}~(\alpha_{i}|\alpha_{i})=0;\\
   &&[ x^{+}_{i}(k),x^{-}_{i}(l)]=-\frac{2}{(\alpha_{i}|\alpha_{i})}\{\alpha_{i}(k+l)
   +k\delta_{k,-l}\mathcal{K}\},\mbox{\emph{if}}~(\alpha_{i}|\alpha_{i})\neq0;\\
   &5)&[x^{\pm}_{i}(k),x^{\pm}_{i}(l)]=0;\\
   &&[x^{\pm}_{i}(k),x^{\pm}_{j}(l)]=0,\mbox{\emph{if}}~a_{ii}=a_{ij}=0,i\neq j;\\
&&[x^{\pm}_{i}(k),[x^{\pm}_{i}(k),x^{\pm}_{j}(l)]]=0,\mbox{\emph{if}}~a_{ii}=0,a_{ij}\neq0,i\neq j;\\
   &&\underbrace{[x^{\pm}_{i}(k),\cdots,[}_{1-a_{ij}}x^{\pm}_{i}(k),x^{\pm}_{j}(l)]\cdots]=0,\mbox{\emph{if}}~a_{ii}\neq0,i\neq j.
\end{eqnarray*}
\end{theorem}
We define formal power series with coefficients from
$\mathfrak{T}(A(m,n))$:
$$\alpha_{i}(z)=\sum_{k\in\mathbb{Z}}\alpha_{i}(k)z^{-k-1},\quad
x^{\pm}_{i}(z) =\sum_{k\in\mathbb{Z}}x^{\pm}_{i}(k)z^{-k-1},$$ then
the defining relations of $\mathfrak{T}(A(m,n))$ can be rewritten in
terms of formal series as follows.

\begin{proposition}\emph{The relations of $\mathfrak{T}(A(m,n))$ can be written as
follows.}
 \begin{eqnarray*}
    &1')& [\mathcal{K},\alpha_{i}(z)]=[\mathcal{K},x^{\pm}_{i}(z)]=0; \\
    &2')& [\alpha_{i}(z),\alpha_{j}(w)]=(\alpha_{i}|\alpha_{j})\partial_{w}\delta(z-w)\mathcal{K};\\
    &3')& [\alpha_{i}(z),x^{\pm}_{j}(w)]=\pm(\alpha_{i}|\alpha_{j})x^{\pm}_{j}(w)\delta(z-w); \\
    &4')& [x^{+}_{i}(z),x^{-}_{j}(w)]=0,\mbox{\emph{if}}~i\neq j   ;\\
    &&[x^{+}_{i}(z),x^{-}_{i}(w)]=-\{(\alpha_{i}(w)
    \delta(z-w)+\partial_{w}\delta(z-w)\mathcal{K}\},\mbox{\emph{if}}~(\alpha_{i}|\alpha_{i})=0\\
    &&[x^{+}_{i}(z),x^{-}_{i}(w)]=-\frac{2}{(\alpha_{i}|\alpha_{i})}\{(\alpha_{i}(w)
    \delta(z-w)+\partial_{w}\delta(z-w)\mathcal{K}\},\mbox{\emph{if}}~(\alpha_{i}|\alpha_{i})\neq0\\
    &5')&[x^{\pm}_{i}(z),x^{\pm}_{i}(w)]=0;\\
    &&[x^{\pm}_{i}(z),x^{\pm}_{j}(w)]=0,\mbox{\emph{if}}~a_{ii}=a_{ij}=0,i\neq j;\\
    &&[x^{\pm}_{i}(z_{1}),[x^{\pm}_{i}(z_{2}),x^{\pm}_{j}(w)]]=0,\mbox{\emph{if}}~a_{ii}=0,a_{ij}\neq0,i\neq j;
    \end{eqnarray*}
    \begin{eqnarray*}
    &&[x^{\pm}_{i}(z_{1}),\cdots,[x^{\pm}_{i}(z_{1-a_{ij}}),x^{\pm}_{j}(w)]\cdots]=0,
    \mbox{\emph{if}}~a_{ii}\neq0,i\neq j.
    \end{eqnarray*}
 \end{proposition}
Here we have used the formal delta function
$$\delta(z-w)=\sum_{n\in\mathbb{Z}}z^{-n-1}w^{n}.
$$
Its derivatives are given by the power series expansions \cite{K}:
$$\partial^{(j)}_{w}\delta(z-w)=i_{z,w}\frac{1}{(z-w)^{j+1}}-i_{w,z}\frac{1}{(-w+z)^{j+1}}.$$
where $\partial^{(j)}_{w}=\partial^{j}_{w}/j!$ and $i_{z,w}$ means
power series expansion in the domain
 $|z|>|w|$. By convention if we write
 a rational function in the variable $z-w$
 it is usually assumed that
 the power series is expanded in the region
 $|z|>|w|$. Finally the equation
 $~~f(z,w)\delta(z-w)=f(z,z)\delta(z-w)$ holds
 when both sides are meaningful.

\section{Vertex representation of $\mathfrak{T}(A(m,n))$}
In this section we will give a representation of the Lie
superalgebra $\mathfrak{T}(A(m,n))$ using both vertex operators and
bosonic fields .

 Let $\varepsilon_{i}$ $(0\leqslant i\leqslant n+m+3)$ be an orthomormal
basis of the vector space $\mathbb{C}^{n+m+4}$ and denote by
$\delta_{i}=\sqrt{-1}\varepsilon_{m+1+i}~(1\leqslant i\leqslant
n+2)$, then the distinguished simple root systems, positive root
systems and longest distinguished root of the Lie superalgebra of
type $A(m,n)$ can be represented in terms of vectors
$\varepsilon_{i}$'s and $\delta_{i}$ 's as follows:
\begin{eqnarray*}
&&\Pi=\big\{\alpha_{1}=\varepsilon_{1}-\varepsilon_{2},\cdots,
,\alpha_{m}=\varepsilon_{m}-\varepsilon_{m+1},\alpha_{m+1}=\varepsilon_{m+1}-\delta_{1},\\
&&\qquad\alpha_{m+2}=\delta_{1}-\delta_{2},\cdots,\alpha_{n+m+1}=\delta_{n}-\delta_{n+1}\big\};\\
&&\triangle_{+}=\big\{\varepsilon_{i}-\varepsilon_{j},\delta_{k}-\delta_{l}|1\leqslant
i<j\leqslant n+1,1\leqslant k<l\leqslant m+1\big\}\\
&&\qquad\cup\big\{\delta_{k}-\varepsilon_{i}|1\leqslant i\leqslant n+1,1\leqslant k\leqslant m+1\big\};\\
&&\theta=\alpha_{1}+\cdots+\alpha_{m+n+1}=\varepsilon_{1}-\delta_{n+1}.\
\end{eqnarray*}

\subsection{Vertex operators}  Let
$\Gamma=\mathbb{Z}\varepsilon_{1}\oplus\cdots\oplus\mathbb{Z}\varepsilon_{m+1}$
and $\mathfrak{h}=\Gamma\otimes_{\mathbb{Z}}\mathbb{C}$. We view
$\mathfrak{h}$ as an abelian Lie algebra and consider the central
extension of its affinization $\widehat{\mathfrak{h}}$, i.e.
$$\widehat{\mathfrak{h}}=\bigoplus_{n\neq 0}\mathbb C\mathfrak{h}\otimes t^n\oplus\mathbb{C}K$$
with the following communication relations:
$$[\alpha(k),\beta(l)]=k(\alpha,\beta)\delta_{k,-l}K,\quad [\widehat{\mathfrak{h}},K]=0$$
where $\alpha(k)=\alpha\otimes t^{k}$ and $\alpha,\beta\in \Gamma;
k,l\in \mathbb{Z}$. This is an infinite dimensional Heisenberg
algebra.

For $i=0,1$, we let $\Gamma_{\overline{i}}=\{\alpha\in
\Gamma|(\alpha,\alpha)\in 2\mathbb{Z}+i\}$, then
$\Gamma=\Gamma_{\overline{0}}\oplus \Gamma_{\overline{1}}$. Let
$F:\Gamma\times \Gamma\rightarrow \{\pm1\}$ be the bimultiplicative
map determined by
$$F(\varepsilon_{i},\varepsilon_{j})
=\left\{
\begin{array}{ll}
1,&\mbox{if}~~i\leqslant j; \\
 -1,&\mbox{if}~~i> j.
\end{array}
\right.$$ Then the map satisfies the following properties:
\begin{eqnarray*}
1)&&F(0,\alpha)=F(\alpha,0)=1,\quad\forall~\alpha\in \Gamma;\\
2)&&F(\alpha,\beta)F(\alpha,\beta+\gamma)=F(\beta,\gamma)F(\alpha,\beta+\gamma),
\quad\forall~\alpha,\beta,\gamma\in \Gamma;\\
3)&&F(\alpha,\beta)F(\beta,\alpha)^{-1}=(-1)^{(\alpha,\beta)+ij},\quad\forall~\alpha\in
\Gamma_{\overline{i}}, \beta\in \Gamma_{\overline{j}}.
\end{eqnarray*}

Let $\mathbb{C}[\Gamma]$ be the vector space spanned by the basis
$\{e^{\gamma}|\gamma\in \Gamma\}$ over $\mathbb{C}$. We define a
twisted group algebra structure on $\mathbb{C}[\Gamma]$ as follows:
$$e^{\alpha}e^{\beta}=F(\alpha,\beta)e^{\alpha+\beta}.$$

We form the tensor space $$V[\Gamma]=\mathbb{C}[\Gamma]\bigotimes
S\big(\oplus_{j<0}(\mathfrak{h}\otimes t^{j})\big),$$ and define the
action of $\widehat{\mathfrak{h}}$ as follows: $K$ acts as the
identity operator,  $\alpha(-k)$ ($k>0$) acts as multiplication by
$\alpha\otimes t^{k}$ for $\alpha\in \Gamma$, and $\alpha(k)$
($k>0$) acts as the derivation of $V[\Gamma]$ defined by the formula
\begin{subequations}
\begin{align}
&\alpha(k)(v\otimes e^{\beta})=k(\alpha,\beta)(v\otimes e^{\beta}),\\
&\alpha(k)(e^{\beta}\otimes\gamma\otimes t^{-l})\\  \nonumber
&=\delta_{k,0}(\alpha,\beta) (e^{\beta}\otimes\gamma\otimes
t^{-l})+k\delta_{k,l}(\alpha,\gamma)e^{\beta}\otimes\gamma\otimes
t^{-k}.
\end{align}
\end{subequations}

The space $V[\Gamma]$ has a natural $\mathbb{Z}_{2}$-gradation:
$V[\Gamma]=V[\Gamma]_{\overline{0}}\oplus V[\Gamma]_{\overline{1}}$,
where $V[\Gamma]_{\overline{0}}$ (resp.$V[\Gamma]_{\overline{1}}$)
is the vector space spanned by $e^{\alpha}\otimes \beta\otimes
t^{-j}$ with $\alpha,\beta\in \Gamma; j\in \mathbb{Z}_{+}$ such that
$(\alpha,\alpha)\in 2\mathbb{Z}$ (resp.$(\alpha,\alpha)\in
2\mathbb{Z}+1$).

For $\alpha\in \Gamma$, we define the vertex operator $Y(\alpha,z)$
as follows:
 $$Y(\alpha,z)=e^{\alpha}z^{\alpha(0)}\mbox{exp}(-\sum_{j<0}\frac{\alpha(j)}{j}z^{-j}),
 \mbox{exp}(-\sum_{j>0}\frac{\alpha(j)}{j}z^{-j})$$
where the operator $z^{\alpha(0)}$ is given by:
\begin{eqnarray*}
 && z^{\alpha(0)}(e^{\beta}\otimes\gamma\otimes t^{-j})
 =z^{(\alpha,\beta)}(e^{\beta}\otimes\gamma\otimes
 t^{-j})
 \end{eqnarray*}
for $\beta,\gamma\in \Gamma;j\in \mathbb{Z}_{+}$ and denote by
$$
X(\alpha,z)=\left\{
\begin{array}{ll}
z^{\frac{(\alpha,\alpha)}{2}}Y(\alpha,z),& \mbox{if}~\alpha\in
\Gamma_{\overline{0}};\\
Y(\alpha,z),& \mbox{if}~\alpha\in \Gamma_{\overline{1}}.
\end{array}
\right.
$$
We expand $X(\alpha,z)$ in $z$
$$X(\alpha,z)=\sum_{j\in\mathbb{Z}}X(\alpha,j)z^{-j-1},$$
where the components $X(\alpha, j)$ are well-defined local operators.
 Similarly for $\alpha\in \Gamma$, we define $$\alpha(z)=\sum_{k\in
\mathbb{Z}}\alpha(k)z^{-k-1}.$$

\begin{lemma} For $\alpha\in \Gamma_{\overline{i}},\beta\in
\Gamma_{\overline{j}}$, one has that
\begin{eqnarray*}
1)&&[Y(\alpha,z),Y(\beta,w)]=0,\quad\mbox{\emph{if}}~(\alpha,\beta)\geqslant0;\\
2)&&[Y(\alpha,z),Y(\beta,w)]=F(\alpha,\beta)Y(\alpha+\beta,z)\delta(z-w),\quad\mbox{\emph{if}}~(\alpha,\beta)=-1;\\
3)&&[\alpha(z),Y(\beta,w)]=(\alpha,\beta)Y(\beta,z)\delta(z-w).
\end{eqnarray*}
\end{lemma}
\begin{proof} The first and second part have been proved in \cite{R}. For
the third part we refer to \cite{X}.
\end{proof}
\begin{corollary}
\begin{eqnarray*}
1) &&
[X(\varepsilon_{i},z),X(\varepsilon_{j}-\varepsilon_{k},w)]=\delta_{ik}F(\varepsilon_{i},
  \varepsilon_{j}-\varepsilon_{k})X(\varepsilon_{j},w)\delta(z-w),\quad  j\neq k;\\
2) &&
[X(\varepsilon_{i},z),X(-\varepsilon_{j},w)]=\delta_{ij}F(\varepsilon_{i},
  -\varepsilon_{j})\partial_{w}\delta(z-w)\\
3)&&[\alpha(z),X(\beta,w)]=(\alpha,\beta)X(\beta,z)\delta(z-w),\quad
\alpha,\beta\in \Gamma;
\end{eqnarray*}
\end{corollary}

\begin{proof} The corollary is direct result of Lemma 3.1.
\end{proof}

\subsection{Bosonic fields}
We introduce  $\overline{c}=\varepsilon_{0}+\delta_{n+2}$ and define
$\beta=\delta_{n+1}+\overline{c}$, then
$\alpha_{0}=\beta-\varepsilon_{1}$. Note that
$(\beta|\beta)=-1,(\beta|\delta_{i})=-\delta_{n+1,i}$. Let $\mathcal
{P}$ be the vector spaces spanned by the set
$\{\overline{c},\delta_{i}|1\leqslant i\leqslant n+1\}$ and
$\mathcal {P}^{*}$ be its dual space. Let $\mathcal{C}=\mathcal
{P}\oplus \mathcal {P}^{*} $ and define the bilinear form on it as
follows: $\mbox{for}~ a,b\in\mathcal {P}$
\begin{eqnarray*}
   && \langle b^{*},a\rangle=-\langle
a,b^{*}\rangle=(a,b);
 \langle b,a\rangle=\langle
a^{*},b^{*}\rangle=0,
\end{eqnarray*}

Let $\mathcal {A}(\mathbb{Z}^{2n+2})$ be the Weyl algebra generated
by $\{u(k)|u\in \mathcal {C},k\in\mathbb{Z}\}$ with the defining
relations
$$u(k)v(l)-u(k)v(l)=\langle u,v\rangle\delta_{k,-l}$$
for $u,v\in \mathcal {C}$ and $ k,l\in\mathbb{Z}$.

The representation space of the algebras $\mathcal
{A}(\mathbb{Z}^{n+1})$ is defined to be the following vector space:

$$\mathfrak{F}=\bigotimes_{a_{i}}\Big(\bigotimes_{k\in\mathbb{Z}_{+}}\mathbb{C}[a_{i}(-k)]
\bigotimes_{k\in\mathbb{Z}_{+}}\mathbb{C}[a^{*}_{i}(-k)]\Big)$$
where $a_{i}$ runs though any basis in $\mathcal{P}$, consisting of,
say $\overline{c}$ and $\delta_{k}$'s. The algebra $\mathcal
{A}(\mathbb{Z}^{2n+2})$ acts on the space by the usual action:
$a(-k)$ acts as creation operators and $a(k)$ as annihilation
operators.

For $u\in \mathcal{C}$, we define the formal power series with
coefficients from the associative algebra $\mathcal
{A}(\mathbb{Z}^{2n+2})$:
$$u(z)=\sum_{k\in\mathbb{Z}}u(k)z^{-k-1}.$$
It is a bosonic field acting on the Fock space $\mathfrak{F}$.

\bigskip

In the following, we will  give a representation of
$\mathfrak{T}(A(m,n))$ on a quotient $\mathfrak V$ of the tensor
space $V[\Gamma]\otimes\mathfrak{F}$:
\begin{align*}
\mathfrak V=V[\Gamma]\otimes\mathfrak{F}/(\sum_{k}:X(\pm\epsilon_1,
-n+k)\overline{c}(n):).
\end{align*}
Therefore the relation $:X(\pm\epsilon_1, z)\overline{c}(z):=0$
holds on $\mathfrak V$. Note that there is a natural homomorphism
from $\mathfrak V$ onto $V[\overline{\Gamma}]\otimes\mathfrak{F}$,
where $\overline{\Gamma}=\Gamma/(\epsilon_0+\delta_{n+2})$.
For simplicity we will use the same symbol to denote the coset
elements in $\mathfrak V$. Observe that there is a
$\mathbb{Z}_{2}-$gradation on this space with the parity given by
$p(e^{\alpha}\otimes x\otimes y)=p(\alpha)$ for $\alpha\in \Gamma,
x\in S(\bigoplus_{j<0}(\mathfrak{h}\otimes t^{j}),y\in
\mathfrak{F}$. The vertex operators $X(\alpha,z),\alpha(z)$ acts on
the first component and the bosonic fields $u(z)$ acts on the second
component. It follows that
$$p(X(\alpha,z))=p(\alpha),\quad p(\alpha(z))=p(u(z))=\overline{0}.$$

For any two fields $a(z),b(w)$ with fixed parity, we define the
normal ordered product by:
\begin{eqnarray*}
:a(z)b(w):&=&a(z)_{+}b(w)-(-1)^{p(a)p(b)}b(w)a(z)_{-}\\
  &=&(-1)^{p(a)p(b)}:b(w)a(z):
\end{eqnarray*} where $a_{\pm}(z)$ is defined as
usual. Based on the normal ordering of two fields, one can define
inductively the normal ordering of more than two fields  ``from right
to left''.

 The following facts are well-known in literature, see
 for example \cite{FLM, JMX}.
\begin{proposition} One has that
\begin{eqnarray*}
1)&&[\alpha(z),\beta(w)]=(\alpha,\beta)\partial_{w}\delta(z-w),\quad \alpha,\beta\in \Gamma\\
2)&&[X(\varepsilon_{i}-\varepsilon_{j},z),X(\varepsilon_{j}-\varepsilon_{i},w)]\\
&&\quad=F(\varepsilon_{i}-\varepsilon_{j},\varepsilon_{j}-\varepsilon_{i})
\big((\varepsilon_{i}-\varepsilon_{j})(z)\delta(z-w)+\partial_{w}\delta(z-w)\big),\\
3)&&:X(\varepsilon_{i},z)X(-\varepsilon_{j},z):=F(\varepsilon_{i},-\varepsilon_{j})
X(\varepsilon_{i}-\varepsilon_{j},z),\quad i\neq  j\\
4)&&:X(-\varepsilon_{j},z)X(\varepsilon_{i},z):=F(-\varepsilon_{j},-\varepsilon_{i})
  X(\varepsilon_{i}-\varepsilon_{j},z),\quad i\neq j\\
5)&&:X(\varepsilon_{i},z)X(-\varepsilon_{i},z):=\varepsilon_{i}(z).
\end{eqnarray*}
\end{proposition}

Furthermore, we define the contraction
 of two fields $a(z),b(w)$  by
$$\underbrace{a(z)b(w)}=a(z)b(w)-:a(z)b(w):.$$

\begin{proposition} \cite{K} Suppose fields $a(z),b(w)$ satisfy the following
equality:
$$[a(z),b(w)]=\sum_{j=0}^{N-1}c^{j}(w)\partial_{w}^{(j)}\delta(z-w),$$
where $N$ is a positive integer and  $c^{j}(w)$ are  formal
distributions in the indeterminate $z$ with value in some algebra
related, then we have that
$$\underbrace{a(z)b(w)}=\sum_{j=0}^{N-1}c^{j}(w)\frac{1}{(z-w)^{j+1}}.$$
\end{proposition}

The following well-known Wick's theorem is useful for
calculating the operator product expansions (OPE) of normally ordered products of free fields.

\begin{theorem} \cite{K} Let $A^{1},A^{2},\cdots,A^{M}$ and
$B^{1},B^{2},\cdots,B^{N}$ be two collections of fields with
definete parity. Suppose these fields satisfy the following
properties:
\begin{eqnarray*}
  1) && [\underbrace{A^{i}B^{j}},Z^{k}]=0,\mbox{for all}~ i,j,k ~\mbox{and}~Z=A~or~B;\\
  2) &&[A^{i}_{\pm},B^{j}_{\pm}]=0,\mbox{for all}~ i,j.
\end{eqnarray*}
then we have that
\begin{align*}
& :A^{1}\cdots A^{M}::B^{1}\cdots B^{N}: \\=&
\sum_{s=0}^m\sum_{i_{1}<\dots<i_{s}\atop
j_{1}\neq\dots\neq
j_{s}}\pm\Big(\underbrace{A^{i_{1}}B^{j_{1}}}\cdots
\underbrace{A^{i_{s}}B^{j_{s}}}:A^{1}\cdots
A^{M}B^{1}\cdots
B^{N}:_{(i_{1},\dots,i_{s},j_{1},\dots,j_{s})}\Big)
\end{align*}
where $m=\min\{M,N\}$
and the subscript $(i_{1},\cdots,i_{s},j_{1},\cdots,j_{s})$ means
the fields $A^{i_{1}},\dots,A^{i_{s}}$, $B^{j_{1}}$, $\dots$, $B^{j_{s}}$
are removed and the sign $\pm$ is obtained by the rule: each
permutation of the adjacent odd fields changes the sign.
\end{theorem}

Now we state the main result in this work.
\begin{theorem} The following map defines a level one representation on
the space  $\mathfrak V$: 
$$x_{i}^{+}(z)\mapsto
\left\{\begin{array}{lll}
\sqrt{-1}:X(-\varepsilon_{1},z)\beta(z):,&i=0;\\
X(\varepsilon_{i}-\varepsilon_{i+1},z),&1\leqslant
i\leqslant m;\\
:X(\varepsilon_{m+1},z)\delta^{*}_{1}(z):,&i=m+1\\
\sqrt{-1}:\delta_{i-m-1}(z)\delta^{*}_{i-m}(z):,&m+2\leqslant
i\leqslant m+n+1.
\end{array}
\right.$$
$$x_{i}^{-}(z)\mapsto
\left\{\begin{array}{lll}
\sqrt{-1}:X(\varepsilon_{1},z)\beta^{*}(z):,& i=0;\\
X(\varepsilon_{i+1}-\varepsilon_{i},z),&1\leqslant
i\leqslant m;\\
:X(-\varepsilon_{m+1},z)\delta_{1}(z):,&i=m+1\\
\sqrt{-1}:\delta_{i-m-1}(z)\delta^{*}_{i-m}(z):,&m+2\leqslant
i\leqslant m+n+1.
\end{array}
\right.$$

$$\alpha_{i}(z)\mapsto
\left\{\begin{array}{lll}
:\beta(z)\beta^{*}(z):-\varepsilon_{1}(z),& i=0;\\
(\varepsilon_{i}-\varepsilon_{i+1})(z),&1\leqslant i\leqslant m;\\
\varepsilon_{m+1}(z)-:\delta_{1}(z)\delta^{*}_{1}(z):,&i=m+1\\
:\delta_{i-m-1}(z)\delta^{*}_{i-m-1}(z):
   -:\delta_{i-m}(z)\delta^{*}_{i-m}(z):,&m+2\leqslant i\leqslant m+n+1.
\end{array}
\right.$$
\end{theorem}
\begin{proof} To prove the theorem, one needs to check that all the field
operators on the right side of above map satisfy  relations $1')$
--- $5')$ listed in Proposition 2.2.

First of all, we check $4')$ and $3')$ with the help of
 Wick's theorem.
\begin{align*}
&[x_{0}^{+}(z),x_{0}^{-}(w)]\\
&=-\big(:\beta(z)\beta^{*}(z):+:X(-\varepsilon_{1},z)X(\varepsilon_{1},z):\big)
\delta(z-w)-\partial_{w}\delta(z-w)\\
&=-\big(\alpha_{0}(z)\delta(z-w)+\partial_{w}\delta(z-w)\cdot 1\big),
\end{align*}
where we have used the fact
$:X(-\varepsilon_{1},z)X(\varepsilon_{1},z):=-\varepsilon_{1}(z)$
and
\begin{eqnarray*}
[\alpha_{0}(z),x_{0}^{\pm}(w)]=0=\pm(\alpha_{0},\alpha_{0})x_{0}^{\pm}(w)\delta(z-w).
\end{eqnarray*}
For $1\leqslant i\leqslant m$, we have by Proposition 3.3 that
\begin{eqnarray*}
[x_{i}^{+}(z),x_{i}^{-}(w)]
&=&-\big((\varepsilon_{i}-\varepsilon_{i+1})(z)\delta(z-w)+\partial_{w}\delta(z-w)\big)\\
&=&-\frac{2}{(\alpha_{i},\alpha_{i})}\big(\alpha_{i}(z)\delta(z-w)+\partial_{w}\delta(z-w)\cdot
1\big).
\end{eqnarray*}
It follows from Corollary 3.2  that
\begin{align*}
&[\alpha_{i}(z),x_{i}^{\pm}(w)]=\pm(\alpha_{i},\alpha_{i})x_{i}^{\pm}(w)\delta(z-w),\\
&[x_{m+1}^{+}(z),x_{m+1}^{-}(w)]\\
&=\big(:\delta_{1}(z)\delta_{1}^{*}(z):-:X(\varepsilon_{m+1},z)X(-\varepsilon_{m+1},z):\big)
\delta(z-w)-\partial_{w}\delta(z-w)\\
&=-\big(\alpha_{m+1}(z)\delta(z-w)+\partial_{w}\delta(z-w)\cdot
1\big),
\end{align*}
and
\begin{eqnarray*}
[\alpha_{m+1}(z),x_{m+1}^{\pm}(w)]=0=\pm(\alpha_{m+1},\alpha_{m+1})x_{m+1}^{\pm}(w)\delta(z-w).
\end{eqnarray*}
For $m+2\leqslant i\leqslant m+n+1$, we have that
\begin{eqnarray*}
&&[x_{i}^{+}(z),x_{i}^{-}(w)]\\
&=&\big(:\delta_{i-m-1}(z)\delta^{*}_{i-m-1}(z):-:\delta_{i-m}(z)\delta^{*}_{i-m}(z):\big)\delta(z-w)+\partial_{w}\delta(z-w)\\
&=&-\frac{2}{(\alpha_{i},\alpha_{i})}\big(\alpha_{i}(z)\delta(z-w)+\partial_{w}\delta(z-w)\cdot
1\big)
\end{eqnarray*}
and
\begin{align*}
[\alpha_{i}(z),x_{i}^{+}(w)]
&=-2\sqrt{-1}:\delta_{i-m-1}(z)\delta^{*}_{i-m}(z):\delta(z-w)\\
&=(\alpha_{i},\alpha_{i})x_{i}^{+}(w)\delta(z-w).\\
[\alpha_{i}(z),x_{i}^{-}(w)]&=-(\alpha_{i},\alpha_{i})x_{i}^{-}(w)\delta(z-w).
\end{align*}
For all $i\neq j$, we have $ [x_{i}^{+}(z),x_{j}^{-}(w)]=0$ and for
any unconnected vertices
\begin{eqnarray*}
[\alpha_{i}(z),x_{j}^{\pm}(w)]=0=\pm(\alpha_{i},\alpha_{j})x_{j}^{\pm}(w)\delta(z-w)
\end{eqnarray*}
All the rest can be checked by straightforward calculation, for
examples
\begin{align*}
[\alpha_{0}(z),x_{1}^{+}(w)]
&=-X(\varepsilon_{1}-\varepsilon_{2},z)\delta(z-w)\\
&=(\alpha_{0},\alpha_{1})x_{1}^{+}(w)\delta(z-w),\\
[\alpha_{m+1}(z),x_{m+2}^{+}(w)]
&=:\delta_{1}(w)\delta_{2}^{*}(w):\delta(z-w)\\
&=(\alpha_{m+1},\alpha_{m+2})x_{m+2}^{+}(w)\delta(z-w),\\
[\alpha_{m+n+1}(z),x_{m+n}^{+}(w)]
&=\sqrt{-1}:\delta_{n-1}(w)\delta_{n}^{*}(w)\delta(z-w)\\
&=(\alpha_{m+n+1},\alpha_{m+n})x_{m+n}^{+}(w)\delta(z-w).
\end{align*}
For the extremal vertices one also has that
\begin{eqnarray*}
[\alpha_{m+n+1}(z),x_{0}^{+}(w)]&=&\sqrt{-1}:X(-\varepsilon_{1},w)\delta_{n+1}(w):\delta(z-w)
\\&=&\sqrt{-1}:X(-\varepsilon_{1},w)\beta(w):\delta(z-w)\\
&=&(\alpha_{m+n+1},\alpha_{0})x_{0}^{+}(w)\delta(z-w),
\end{eqnarray*}
where we have used the fact that $:X(-\varepsilon_1,
w)\overline{c}(w):=0$ and others can be proved similarly.

Secondly, we can check $2')$ case by case by using Proposition 3.3
1) and we include the following examples
\begin{eqnarray*}
&&[\alpha_{0}(z),\alpha_{0}(w)]=0=(\alpha_{0},\alpha_{0})\partial_{w}\delta(z-w)\cdot1\\
&&[\alpha_{0}(z),\alpha_{1}(w)]=-\partial_{w}\delta(z-w)=(\alpha_{0},\alpha_{1})\partial_{w}\delta(z-w)\cdot1\\
&&[\alpha_{0}(z),\alpha_{m+n+1}(w)]=\partial_{w}\delta(z-w)=(\alpha_{0},\alpha_{m+1})\partial_{w}\delta(z-w)\cdot1
\end{eqnarray*}

Finally, we proceed to check the Serre relations. It is easy to
verify that $[x_{i}^{\pm}(z),x_{i}^{\pm}(w)]=0$ for $0\leqslant
i\leqslant m+n+1$ and $[x_{i}^{\pm}(z),x_{j}^{\pm}(w)]=0$ for $
i\neq j,a_{ij}=0$. The rest can be checked directly:
\begin{eqnarray*}
  &&[x_{0}^{+}(z_{1}),[x_{0}^{+}(z_{2}),x_{1}^{+}(w)]]\\
  &=&-[:X(-\varepsilon_{1},z_{1})\beta(z_{1}):,[:X(-\varepsilon_{1},z_{2})\beta(z_{2}):,X(\varepsilon_{1}-\varepsilon_{2},w)]]\\
   &=&-[:X(-\varepsilon_{1},z_{1})\beta(z_{1}):,X(-\varepsilon_{2},w)]\delta(z_{2}-w)\\
   &=&0,
\end{eqnarray*}
\begin{eqnarray*}
  &&[x_{0}^{+}(z_{1}),[x_{0}^{+}(z_{2}),x_{m+n+1}^{+}(w)]]\\
  &=&-\sqrt{-1}[:X(-\varepsilon_{1},z_{1})\beta(z_{1}):,[:X(-\varepsilon_{1},z_{2})\beta(z_{2}):,:\delta_{n}(w)\delta^{*}_{n+1}(w):]]\\
   &=&-\sqrt{-1}[:X(-\varepsilon_{1},z_{1})\beta(z_{1}):,:X(-\varepsilon_{1},w)\delta_{n}(w):]\delta(z_{2}-w)\\
   &=&0,
\end{eqnarray*}
\begin{eqnarray*}
  &&[x_{m+1}^{+}(z_{1}),[x_{m+1}^{+}(z_{2}),x_{m}^{+}(w)]]\\
  &=&[:X(\varepsilon_{m+1},z_{1})\delta^{*}_{1}(z_{1}):,[:X(\varepsilon_{m+1},z_{2})\delta^{*}_{1}(z_{2}):,X(\varepsilon_{m}-\varepsilon_{m+1},w)]]\\
   &=&[:X(\varepsilon_{m+1},z_{1})\delta^{*}_{1}(z_{1}):,;X(\varepsilon_{m},w)\delta_{1}^{*}(w):]\delta(z_{2}-w)\\
   &=&0,
\end{eqnarray*}
\begin{eqnarray*}
  &&[x_{m+1}^{+}(z_{1}),[x_{m+1}^{+}(z_{2}),x_{m+2}^{+}(w)]]\\
  &=&\sqrt{-1}[:X(\varepsilon_{m+1},z_{1})\delta^{*}_{1}(z_{1}):,[:X(\varepsilon_{m+1},z_{2})\delta^{*}_{1}(z_{2}):
  ,:\delta_{1}(w)\delta^{*}_{2}(w):]]\\
   &=&-\sqrt{-1}[:X(\varepsilon_{m+1},z_{1})\delta^{*}_{1}(z_{1}):,;X(\varepsilon_{m+1},w)\delta_{2}^{*}(w):]\delta(z_{2}-w)\\
   &=&0.
\end{eqnarray*}
The remaining relations follow similarly by Wick's theorem or
Corollary 3.2.  This completes the proof of the theorem.
\end{proof}


\begin{thebibliography}{ABCD}
\bibitem{R} S. Eswara Rao,  \emph{Representation of toroidal
general linear superalgebras}, Commun. Alg. 42 (2013), 2476-2507.
\bibitem{RM} S. Eswara Rao, R. V. Moody,  {\em Vertex representations for N-toroidal
Lie algebras and a generalization of the Virasoro algebra}, Commun.
Math. Phys., 159 (1994), 239--264.
\bibitem{BCG} S. Bhargava, H. Chen, Y. Gao, {\em A family of representations of the Lie superalgebra $\widehat{gl}_{1|l-1}(C_q)$}, J. Algebra 386 (2013), 61--76.
\bibitem{FF} A. J. Feingold,  I. B. Frenkel, {\em Classical affine algebras}, Adv. Math. 56 (1985), 117--172.
\bibitem{FJW} I. B. Frenkel, N. Jing, W. Wang, \emph{Vertex representations via finite
groups and the MaKay correspondence}, Int. Math. Res. Notices
4 (2000), 195--222.
\bibitem{FLM}  I. B.  Frenkel,  J. Lepowsky,  A. Meuraman, {\em Vertex Operator algebras
and the Monster}, Academic Press, Boston, 1988.
\bibitem{IK} K. Iohara, Y. Koga, {\em Central extensions of Lie
superalgebras}, Commun. Math. Helv. 76 (2001), 110--154.
\bibitem{JMe} C. Jiang, D. Meng, {\em Vertex
representations for the $\nu+1$-toroidal Lie algebra of type
$B_{l}$}, J. Alg. 246 (2001), 564--593.
\bibitem{JM} N. Jing, K. C. Misra, {\em Fermionic realization of toroidal Lie algebras of classical types}, J. Alg. 324 (2010), 183--194.
\bibitem{JMT} N. Jing, K. C. Misra, S. Tan, {\em Bosonic realizations of higher level toroidal Lie algebras}, Pacific. J. Math. 219 (2005), 285--302.
\bibitem{JMX} N. Jing, K. C. Misra, C. Xu, {\em Bosonic realization of toroidal Lie algebras of classical types}, Proc. AMS 137 (2009), 3609--3618.
 \bibitem{JX} N. Jing, C. Xu, \emph{Toroidal Lie superalgebras and free field representations}, Contemp. Math.,
 to appear, arXiv:1308.2826.
\bibitem{K} V. G.  Kac, {\em Vertex algebras for beginners}, Univ. Lect. Ser. 10, AMS, Providence, 1997.
\bibitem{KWk} V. G. Kac, M. Wakimoto, {\em Integrable highest
weight modules over affine superalgebras and Appell's function},
Commun. Math. Phys. 215 (2001), 631--682.
\bibitem{KWn} V. G. Kac, W. Wang,
{\em Vertex operator superalgebras and their representations}, Contemp. Math., 175 (1994) 161--191.
\bibitem{L} M. Lau, {\em Representations of multiloop algebras}, Pacific J. Math. 245 (2010), 167--184.
\bibitem{LH} D. Liu, N. Hu, {\em Vertex representations for toroidal Lie algebra of
type $G$}, J. Pure Appl. Alg., 198 (2005), 257--279
\bibitem{MRY} R. V. Moody, S. E. Rao, T. Yokonuma, {\em Toroidal Lie algebras and vertex representations}, Geom. Ded. 35 (1990), 283--307.
\bibitem{RZ} S.~E. Rao,  K. Zhao, \emph{On integrable representations for toroidal Lie
superalgebras}, Contemp. Math. 343 (2004), 243--261.
\bibitem{T} S. Tan, {\em Vertex operator representations for toroidal Lie algebra of type $B_l$}, Commun. Alg. 27 (1999), 3593--3618.
\bibitem{X} X. Xu, {\em Introduction to vertex operator superalgebras and their modules}, Kluwer Academic Publishers, Dordrecht, 1998.
\end{thebibliography}
\end{document}